\documentclass[a4paper,11pt,oneside]{amsart}
\usepackage[english]{babel}
\usepackage{amsaddr}
\usepackage{color}
\usepackage{enumerate}
\usepackage{mathrsfs}
\usepackage{xcomment}
\usepackage{comment}

\RequirePackage[normalem]{ulem} 
\RequirePackage{color}\definecolor{RED}{rgb}{1,0,0}\definecolor{BLUE}{rgb}{0,0,1} 

\usepackage{amsmath}
\usepackage{amssymb}
\usepackage{amsthm}
\newcommand{\GS}{\mathcal{G}_\lambda}
\newcommand{\GU}{\mathcal{G}_\lambda (u)}

\newcommand{\nint}{\int_{\RN}}
\newcommand{\RN}{\R\sp\N}
\newcommand{\R}{\mathbb{R}}
\newcommand{\C}{\mathbb{C}}

\newcommand{\X}{H\sp 1 (\RN;\C)}

\renewcommand{\RN}{\R}

\renewcommand{\nint}{\int_{-\infty}^{+\infty}}

\newcommand{\no}[1]{\|#1\|}

\newcommand{\var}{\varepsilon}
\newcommand{\pt}{\partial}
\excludecomment{TAKEOUT}
\includecomment{TAKEIN}
\newcommand{\om}{\omega}
\newcommand{\ra}{\rightarrow}
\newcommand{\dist}{\mathrm{dist}}
\newcommand{\wra}{\rightharpoonup}
\theoremstyle{definition}

\newtheorem{remark}{Remark}[section]
\newtheorem{corollary}{Corollary}[section]
\newtheorem{theo}{Theorem}[]
\newtheorem{lemma}{Lemma}[section]
\newtheorem{proposition}{Proposition}[section]

\numberwithin{equation}{section}
\author[]{Daniele Garrisi and Vladimir Georgiev }
\address{Universit\`a degli Studi di Pisa, Dipartimento di Matematica
"Leonida Tonelli"}
\address{Inha University, College of Mathematics Education}
\title[Uniqueness of the ground state]
{Orbital stability and uniqueness of the ground state for NLS in dimension one}
\thanks{%
The first author was supported by INHA UNIVERSITY Research Grant
through the project number 51747-01 titled "Stability in non-linear
evolution equations". The second author was supported by University of 
Pisa, 
project no. PRA-2016-41 "Fenomeni singolari in problemi deterministici
e stocastici ed applicazioni"; by INDAM, GNAMPA - Gruppo Nazionale per 
l'Analisi Matematica, la Probabilit\`{a} e le loro Applicazioni and by 
Institute of Mathematics and Informatics, Bulgarian Academy of Sciences.}
\begin{document}
\maketitle
\begin{abstract}
We prove that standing-waves solutions to the non-linear
Schr\"odinger equation in dimension one whose profiles can be obtained
as minima of the energy over the mass, are orbitally stable and non-degenerate,
provided the non-linear term $ G $ satisfies a Euler differential inequality.
When the non-linear term $ G $ is a combined pure power-type, then
there is only one positive, symmetric minimum of the energy constrained
to the constant mass.
\end{abstract}
\thispagestyle{empty}
\section{Introduction}
The purpose of this paper is to prove the orbital stability
of solitary-wave solutions to a non-linear Schr\"odinger equation
\[
\tag{NLS}
\label{eq.NLS}
i\partial_t \phi(t,x) + \Delta_x \phi(t,x) - f(\phi(t,x)) = 0,\quad
\phi\colon\R_t\times\R\sp n_x\ra\C
\]
in dimension $ n = 1 $ for general class of nonlinear functions
$ f $ such that $ f\colon\C\ra\C $ is $ C\sp 1 $ and
\begin{equation}\label{eq.h1}
f(\overline{s}) = \overline{f(s)}, \quad f(zs) = zf(s),\quad
\forall z\in \C\text{ such that } |z| = 1.
\end{equation}
From \eqref{eq.h1}, if $ s $ is a real number, then $ f(s) $ is
a real number. We denote by $ g\colon\R\to\R $ the restriction of $ f $ to
$ \R $. From the second equality of \eqref{eq.h1}, $ g $ is an
even function. Let $ G $ be the primitive of $ g $ such that $ G(0) = 0 $.
We define for every complex number $ s $
\[
F(s) := G(|s|).
\]
A solitary-wave is a solution to \eqref{eq.NLS} of
the form
\begin{equation}\label{eq.I1}
   \phi(t,x) = e\sp{i\om t} R(x),\quad
(t,x)\in [0,\infty)\times\RN
\end{equation}
so that the gauge invariance condition \eqref{eq.h1} implies that
\[
\tag{SW}
\label{eq.SW}
\phi(t,x) = z e\sp{i(v\cdot x - t|v|\sp 2)}
e\sp{i\om t} R(x - 2tv),\quad
(t,x)\in [0,\infty)\times\RN
\]
is also a solution to \eqref{eq.NLS}
for any $ v $ in $ \RN $ and any complex number $ z $ such that $ |z| = 1 $.
The profile $ R $ is a real-valued function. If $ \phi $ in \eqref{eq.SW}
is a solution to \eqref{eq.NLS}, then $ R $
satisfies the differential equation
\begin{equation}
\label{eq.E}
-R''(x) +f(R(x)) + \om R(x) = 0.
\end{equation}
In this paper we address solutions to the equation above
which can be obtained as minima of the  energy functional
\[
E\colon \X\ra\R,\quad
E(u) := \frac{1}{2}\nint |u'(x)|\sp 2 dx + \nint F(u(x))dx
\]
under the constraint
\[
S(\lambda) := \{u\in H\sp 1\mid M(u) = \lambda\},
\]
where $ M(u) := \nint |u(x)|^2dx $.
We will assume that the non-linearity satisfies conditions which
guarantee the global well-posedness of the initial value problem
of \eqref{eq.NLS} in $ H^1 $; that is, given an initial datum
$ u_0\in\X $, there exists a unique solution
$ \phi(t,x)\in C([0,+\infty);\X) $
to the Schr\"odinger equation
such that $ \phi(0,x) = u_0 (x) $.

The global well-posedness determines a one-parameter family of operators
$ U_t $ on $ \X $.

To a real number $ \lambda > 0 $, we associate two
subsets of $ H^1 (\RN;\C) $:
\begin{gather*}
\mathcal{G}_\lambda :=
\{u\in\X\mid E(u) = \inf_{S(\lambda)} E\}\\
\mathcal{G}_\lambda (u) := \{zu(\cdot + y)\mid (z,y)\in S^1\times\R\}.
\end{gather*}
The first one is called \textsl{ground state}.
The second set is a subset of the ground state; if $ u(x) = R(x) $,
then $ \mathcal{G}_\lambda (R) $ contains
the orbit of $ R $, that is, the point $ U_t (R) $ for every $ t\geq 0 $.

On $ \X $, we consider the distance given by the scalar product
\[
(u,w)_{\X} :=
\text{Re}\nint u\overline{w}(x) dx +
\text{Re}\nint\nabla u\cdot\nabla\overline{w}(x) dx.
\]
A set $ S\subseteq\X $ is said \textsl{stable}
if for every $ \var > 0 $, there exists $ \delta > 0 $ such that
\[
\dist(u,S) < \delta\implies\dist(U_t (u),S)
\]
for every $ t\geq 0 $.
One of the first result of stability is the work of
T.~Cazenave and P.~L.~Lions in 1982, \cite{CL82},
where $ f $ is a pure power function.
Extensions
to more general non-linearities have been obtained in
\cite{Wei86} and \cite{BBGM07,Shi14}. However, while
in \cite{CL82} the stability of both $ \mathcal{G}_\lambda $ and
$ \mathcal{G}_\lambda (u) $ has been proved, in \cite{BBGM07,Shi14}
only the stability of the ground state is proved.
The pure power case
\[
f(s) = -|s|\sp{p - 2} s
\]
is very special as it exhibits
the rescaling invariance $ f(ts) = t\sp{p - 1} f(s) $ for every
$ t\geq 0 $.
As a consequence, $ \GS = \GU $
for every $ u $ in $ \GS $.
In fact, it is possible to give a precise description of
an element $ u $ of the ground state:
\[
u(x) = \pm z\om\sp{\frac{1}{p - 1}} R_1 (\omega\sp{1/2} (x + y)),
\quad \om\sp{\frac{5 - p}{2(p - 1)}} \no{R_1}_{L\sp 2}\sp 2 = \lambda
\]
where $ R_1 $ is the unique positive solution to
\eqref{eq.E} when $ \om = 1 $,
$ y $ is in $ \RN $, and $ z $ is complex with $ |z| = 1 $.
Therefore, the set $ \GU $
is stable because the ground
state is stable. When more general non-linearities are considered,
the rescaling property fails, and it is not clear anymore whether
the equality $ \GS = \GU $ holds. We list our assumptions:
\[
\tag{G1}
\label{G1}
\text{ there exists } s_0\in\R\text{ such that } G(s_0) < 0
\]
there exist $ C $, $ p,q,p_* $ and $ s_* $ such that
\[
\tag{G2}
\label{G2}
\begin{split}
|G'(s)|\leq C(|s|\sp{p - 1} + |s|\sp{q - 1}),\quad\forall s\in\R\\
-C|s|\sp{p_* - 1}\leq G'(s),\quad\forall s\geq s_*
\end{split}
\]
where $ 2 < p_* < 6 $ and $ 2 < p\leq q $.


\begin{theo}[Orbital stability of $ \GS $]
\label{thm.ground-state}
Suppose that \eqref{G1} and \eqref{G2} hold. Then,
there exists $ \lambda_* \geq 0 $ such that for every
$ \lambda > \lambda_* $, the functional $ E $ has a minimum, and
$ \mathcal{G}_\lambda $ is orbitally stable.
\end{theo}
A proof of the stability of $ \mathcal{G}_\lambda $ has been made
in \cite{BBGM07} in dimension $ n\geq 3 $ and in \cite{Shi14,BS11}.
Here we present a few improvements with respect to the assumptions made
in the quoted references. Firstly, we do not use (at this point)
the growth condition \eqref{G4}, required \cite[$ \mathrm{F}_p ' $]{BBGM07}
to obtain the splitting property \eqref{prop.1.3} of
Proposition~\ref{prop.1},
which follows directly from \cite{BL83}; secondly, \eqref{G2} weakens
\cite[F4]{Shi14}, where $ s_* $ is set to zero. Instead, we allow
the non-linearity to be critical nearby the origin.
We prove the orbital stability of $ \GS $ with a version of the
Concentration-Compactness Lemma of P.~L.~Lions,
\cite{Lio84a}, introduced by V.~Benci and D.~Fortunato in \cite{BF14} where the
classic definitions of Concentration, Compactness and Vanishing are expressed in
terms of weak convergence, instead of the Concentration Function used in
\cite{Lio84a}.

Concentration. There exists a subsequence $ (u_{n_k}) $, a sequence
$ (y_k) $ and $ u $ such that
\begin{equation}
\label{eq.C}
\tag{C}
u_{n_k} (\cdot + y_{n_k})\ra u\text{ in } H\sp 1 (\R;\C).
\end{equation}
Dichotomy. There exists a subsequence $ (u_{n_k}) $ and $ (y_k) $ such that
\begin{equation}
\label{eq.D}
\tag{D}
u_{n_k} (\cdot + y_k)\wra u\text{ in } H\sp 1 (\R;\C)
\end{equation}
for some $ u $ such that
$ 0 < \no{u}_{H\sp 1} < \text{liminf}_{k\ra +\infty} \no{u_{n_k}}_{H\sp 1} $.

Vanishing
\begin{equation}
\label{eq.V}
\tag{V}
u_{n_k}(\cdot + y_k)\wra u\implies u = 0.
\end{equation}
The functional $ E $ is defined on $ \X $ instead of real-valued functions.
This perspective of the minimization problem has the value of highlighting
features of the minima which are essential in the proof of the stability
of the other set, $ \GS(u) $: if $ u $ is a minimum, then $ u = zR $, where
$ R $ is a real-valued minimum and $ z $ a complex number with $ |z| = 1 $,
\eqref{lem.minima.4} of Lemma~\ref{lem.minima}. This fact relies on the
Convex Inequality for the Gradient, \cite[Lemma~7.8]{LL01}.
In the next assumption $ G $ is $ C^2 $ and
\[
\tag{G3}
\label{G3}
12G(R(x)) - 7R(x)G'(R(x)) + R(x)^2 G''(R(x))\geq 0
\]
for every solution $ R $ of \eqref{eq.E} and $ x\in\R $.
\[
\label{G4}
\tag{G4}
|G''(s)|\leq C(|s|\sp{p - 2} + |s|\sp{q - 2})\text{ for every } s\in\R
\]
and $ p,q $ as in \eqref{G2}.
We denote with $ H^1 _r (\R) $ the space of real-valued $ H^1 $ functions
which are even on $ \R $. Let $ \lambda_* $ be as in
Theorem~\ref{thm.ground-state}.
\begin{theo}
\label{thm.non-degeneracy}
Suppose that \eqref{G1}, \eqref{G2}, \eqref{G3} and \eqref{G4} hold.
Then, for $ \lambda > \lambda_* $, minima of $ E $ on
$ S(\lambda)\cap H^1 _r $ are non-degenerate.
\end{theo}
Our work presents some changes with respect to the one of M.~Weinstein,
\cite{Wei86} for the one-dimensional case. As we mentioned earlier in the
introduction, in order to have stability the condition
\[
\tag{B3}
\label{B3}
\nint
\left[f(R^2) + 2R^2 f'(R^2) |R'(x)|^2\right]dx\neq 0
\]
was required (in his notation $ f(s)s = G'(s) $). In \eqref{G3}
we offer a different approach, as we prescribe a condition on
the non-linearity rather than on a solution to \eqref{eq.E}.

The non-degeneracy implies that the set $ \GS\cap H^1 _r $ is finite.
This should be compared with the pure-power case, where $ \GS\cap H^1 _r $
consists of exactly two functions. Consequently, under the same
assumptions as the theorem above and adding the assumption

\begin{theo}
\label{thm.sw}
Then the set $ \GS(u) $ is stable for every $ u\in\GS $.
\end{theo}
Finally, we show that under an additional assumption, a uniqueness
condition holds, just like the pure-power case. Then
$ \GS = \GS(u) $, Corollary~\ref{cor.uniqueness}.
For $ \omega > 0 $, we define
\begin{equation}
\label{eq.H}
V(s) := -\frac{2G(s)}{s^2}
\end{equation}
and
\begin{equation}
\label{eq.R}
R_* (\omega) := \inf\{s > 0\mid \omega = V(s)\}
\end{equation}
whenever the set on the right is non-empty.
\begin{enumerate}[(G1)]
\setcounter{enumi}{4}
\item\label{G5}
The set $ \mathscr{A} = \{\omega\mid V'(R_* (\omega)) > 0\} $ is an interval.
\end{enumerate}
\begin{theo}[Uniqueness]
\label{thm.uniqueness}
If (G1-5) hold, then $ \GS\cap H^1 _r $ consists of exactly two functions,
$ R_+ $ and $ R_- $. The first is positive while $ R_- = -R_+ $.
\end{theo}
Both the proofs of the uniqueness and the stability of $ \GS $
rely on the function $ d(\omega) $ defined by W.~Strauss in \cite{SS85}
and \cite{GSS87}. Condition which guarantees the
stability $ d''(\omega)\geq 0 $ follows from \eqref{G3}, and
Lemma~\ref{lem.non-degeneracy}. Here, we deduce the stability
of $ \GS(u) $ from the fact that the set $ H^1 _r \cap\GS $
is finite, rather than checking the assumptions of \cite{GSS87} directly.

We define
\[
L(s) = 12G(s) - 7sG'(s) + s^2 G''(s).
\]
When $ L = 0 $ on $ (0,+\infty) $, then $ F $ satisfies a Euler
equation whose solutions are linear combinations of $ s^2 $ and $ s^6 $.
If $ G $ is a pure-power $ -as^p $ with $ a > 0 $, then
\[
L(s) = a(p - 2)(6 - p)s^p
\]
which is strictly positive if $ p $ is sub-critical. Therefore
\eqref{G3} can be interpreted as a sub-critical assumption.
However, this interpretation fails as we consider sub-critical pure-power
combined non-linearities as
\begin{equation}
\label{eq.cp}
F(s) = -a|s|^p + b|s|^q,\quad p < q
\end{equation}
where
\[
L(s) = a(p - 2)(6 - p)s^p - b(q - 2)(6 - q)s^q
\]
and can be negative on $ (0,+\infty) $.
However, \eqref{G3} prescribes the behaviour
of $ L $ \textsl{only} on the union of the images of the solutions of
\eqref{eq.E} (for arbitrary $ \omega $). In fact, it turns out that
\eqref{eq.cp} does satisfy \eqref{G3}.

The paper is organized as follows. In \S\ref{sec.cc} we show
the Concentration-Compactness behaviour of minimizing sequences;
in \S\ref{sec.non-degenerate} we discuss the non-degeneracy of minima,
in \S\ref{sec.stability} the stability of the two sets $ \GS $ and
$ \GS(u) $, in \S\ref{sec.uniqueness} the uniqueness of positive, even
solutions in $ \GS $. In \S\ref{sec.cp} we show that \eqref{eq.cp}
satisfies all the assumptions mentioned above.
\section{Properties of the functional $ E $}
\label{sec.cc}
In Lemma~\ref{prop.split} we show that $ G $ can be written as sum of two
terms $ G_1 $ and $ G_2 $ which satisfy estimates \eqref{G2} and \eqref{G4}
having only a single power on the right term.
Since all the properties we will prove are well-behaved with respect to
linear combination, we will assume that in \eqref{G2} and \eqref{G4} there
is only the power $ p $.

Some of the properties listed in the next proposition have been
already thoroughly proved in \cite{BBGM07} in dimension
$ n\geq 3 $. We fill the details of the proof in the dimension
$ n = 1 $. Throughout this section we will assume that \eqref{G2}
holds.
\begin{proposition}
\label{prop1}
The functional $ E $ satisfies the following properties:
\begin{enumerate}[(i)]
\item
\label{prop1:2}
given $ e,\lambda > 0 $, there exists $ C(e,\lambda) $ such that
\[
E(u)\leq e\text{ and } M(u)\geq\lambda\implies\no{u}_{H^1} \leq C.
\]
Then minimizing sequences of $ E $ over $ S(\lambda) $ are bounded
\item
\label{prop.1.3}
if \eqref{G2} holds, given a weakly converging sequence
$ u_n\wra u $
\begin{align*}
E(u_n) &= E(u_n - u) + E(u) + o(1)\\
M(u_n) &= M(u_n - u) + M(u) + o(1)
\end{align*}
\item
\label{prop1:8}
given a bounded sequence $ (u_n) $ in $ H\sp 1 $ and
a sequence $ (\lambda_n) $ such that $ \lambda_n\to\lambda $, then
\[
\lim_{n\ra +\infty} \big(E(\lambda_n u_n) - E(\lambda u_n)\big) = 0
\]
\item
\label{prop1:4}
$ E $ is bounded from below on $ S(\lambda) $.
\end{enumerate}
\end{proposition}
\begin{proof}
\,

\eqref{prop1:2} and \eqref{prop1:4}. From the Sobolev-Gagliardo-Nirenberg
inequality there exists a $ S\in\R $ such that
\begin{equation}
\label{eq.SGN}
\no{u}_{L\sp d}\sp d\leq S\sp d
\no{u}_{L\sp 2}\sp{\frac{d + 2}{2}}
\no{u'}_{L\sp 2}\sp{\frac{d - 2}{2}}
\end{equation}
for every $ d\geq 2 $. We set $ A := \{|u|\geq s_*\} $. On $ A $,
$ F(u(x)) $ is bounded from below by $ -C|u(x)|^p $ or
$ -Cs_* ^{p - p_*} |u(x)|^{p_*} $, depending on whether
$ p\leq p_* $. On $ A\sp c $, we use $ -C|u(x)|\sp{p_*} $.
In any case, from \eqref{G2}
\[
\nint F(u(x))dx\geq -C\max\{1,s_* ^{p - p_*}\} \no{u}_{L^d} ^d
\]
where $ d < 6 $. Then, from \eqref{eq.SGN}
\begin{equation}
\label{eq.14}
\begin{split}
2e\geq 2E(u)\geq\no{u'}_{L\sp 2}\sp 2 -
C'\lambda\sp{\frac{d + 2}{4}}
\no{u'}_{L\sp 2}\sp{\frac{d - 2}{2}}.
\end{split}
\end{equation}
Then $ C(e,\lambda) $ exists because $ (d - 2)/2 < 2 $. Then minimizing
sequences are bounded
\eqref{prop.1.3}. We refer to the paper of H.~Brezis and E.~Lieb~\cite{BL83}.

\eqref{prop1:8}. We write the proof only for the non-linear part
$ \int F(u) dx $, for which we use the same notation $ E $. We define
\[
\begin{split}
k_n (x) &:= F(\lambda_n u_n(x)) - F (\lambda u_n(x)) \\
&= \int_0\sp 1 F'(t\lambda u_n(x) + (1 - t)\lambda_n u_n(x))
(\lambda_n - \lambda)u_n(x) dt.
\end{split}
\]
Then
\[
\begin{split}
|k_n (x)|&\leq C |\lambda_n - \lambda|
\int_0\sp 1 |t\lambda u_n(x) + (1 - t)\lambda_n u_n(x)|\sp{p - 1} |u_n(x)| dt\\
&\leq 2\sp{p - 2} C |\lambda_n - \lambda|
\int_0\sp 1 \big(|t\lambda u_n(x)|\sp{p - 1} +
|(1 - t)\lambda_n u_n(x)|\sp{p - 1}\big) |u_n(x)|dt\\
&\leq 2\sp{p - 2} C|\lambda_n - \lambda|
(|\lambda|\sp{p - 1} + |\lambda_n|\sp{p - 1}) |u_n(x)|\sp{p - 1}.
\end{split}
\]
By integrating on $ \R $, we obtain
\[
\begin{split}
|E(\lambda_n u_n) - E(u_n)|&\leq\nint |k_n(x)| dx\\
&\leq 2\sp{p - 2} C|\lambda_n - \lambda|
(|\lambda|\sp{p - 1} + |\lambda_n|\sp{p - 1}) \no{u_n}_{L\sp p}\sp p.
\end{split}
\]
Since $ (u_n) $ is bounded in $ H^1 $, as we take the limit
as $ n\to +\infty $, we obtain the conclusion.
\end{proof}
We are then allowed to define
\[
I(\lambda) := \inf_{S(\lambda)} E.
\]
\begin{proposition}
\label{prop.1}
The function $ I $ satisfies the following properties:
\begin{enumerate}[(i)]
\item
\label{prop1:5}
the function $ I $ is non-positive
\item
\label{prop1:6}
for every $ \vartheta\geq 1 $ and $ \lambda > 0 $, there holds
\[
I(\vartheta\lambda)\leq\vartheta I(\lambda).
\]
If equality holds, then either $ \vartheta = 1 $ or $ \vartheta > 1 $
and $ I(\lambda) = 0 $
\item
\label{prop.1.7}
there exists $ \lambda_* > 0 $ such that
\[
I < 0\text{ on } (\lambda_*,+\infty),\quad I = 0\text{ on } (0,\lambda_*].
\]
If $ \lambda\leq\lambda_* $, then $ \GS $ is empty.
\end{enumerate}
\end{proposition}
\begin{proof}
\,

\eqref{prop1:5}. The proof of this fact follows from
\cite[Lemma~2.3]{Shi14}.

\eqref{prop1:6}.
Such property of $ I $ has been proved in
\cite[Lemma~3.2]{Shi14} and \cite[Proposition~15]{BBGM07} using
rescalings. However, in both references it is assumed
that $ E $ achieves its infimum on $ S(\lambda) $.
Here, we just apply the same rescaling to a minimizing
sequence $ (u_n) $ over $ S(\lambda) $
\[
u_{n,\vartheta} (x) = u_n (\vartheta\sp{-1} x).
\]
Clearly, $ u_n\in S(\vartheta\lambda) $. Then
\begin{equation}
\label{prop.1.eq.1}
\begin{split}
I(\vartheta\lambda)\leq E(u_{n,\vartheta}) &= \vartheta\left(
\frac{\vartheta\sp{-2}}{2} \nint |u_n'|\sp 2 dx +
\nint F(u_n)dx\right) \\
&= \vartheta\left(E(u_n) -
\frac{1 - \vartheta\sp{-2}}{2}\no{u_n'}_{L\sp 2}\sp 2\right)
\\
&\leq \vartheta I(\lambda) -
\frac{\vartheta (1 - \vartheta\sp{-2})}{2}
\cdot\no{u_n'}_{L\sp 2}\sp 2\leq\vartheta I(\lambda).
\end{split}
\end{equation}
Clearly, if equality holds and $ \vartheta > 1 $, then the
sequence of gradients converges to zero. From \eqref{eq.f4} and
\eqref{eq.SGN}, we obtain
\[
\nint F(u_n)dx\to 0.
\]
Therefore, $ I(\lambda) = 0 $.

\eqref{prop.1.7}.
From \cite[Lemma~5]{BBGM07} it follows that there exists
$ \lambda_1 $ such that $ I(\lambda_1) < 0 $. Here we use the assumption
\eqref{G1}. Now, suppose that there exists $ \lambda_0 $ such that
$ I(\lambda_0) = 0 $. If $ \lambda' \leq \lambda_0 $, then
\[
0 = I(\lambda') =
I(\lambda'/\lambda_0\cdot\lambda_0)\leq (\lambda'/\lambda_0)
I(\lambda_0)\leq 0.
\]
The first inequality follows from (ii), while the last inequality follows
from \eqref{prop1:5}. Therefore, the set
$ \{\lambda\mid I(\lambda) = 0\} $ is bounded or is equal to
$ (0,+\infty) $. The second case is ruled about by $ \lambda_1 $.
On the first case, we define
\[
\lambda_* := \sup\{\lambda\mid I(\lambda) = 0\}.
\]
Since $ I $ is continuous (from \cite[Lemma~2.3]{Shi14}),
$ I(\lambda_*) = 0 $. Now, we consider the case $ \lambda \leq\lambda_1 $.
Let $ u\in\GS $ be a minimum. We apply to $ u $ the same rescaling
as in \eqref{prop.1.eq.1}. For every $ \vartheta $ the endpoints of
the inequalities are zero, then $ \no{u'}_2\sp 2 = 0 $
which gives $ u = 0 $, and obtain a contradiction with $ \lambda > 0 $.
\end{proof}
We define $ J_k $ the open interval $ (k,k + 1) $.
The result of the next lemma is well-known from
\cite[Lemma~I.1]{Lio84a}: if a sequence $ (u_n) $ vanishes, then
all the $ L\sp d $ norms converge to zero. In \cite{Lio84a} they show that if
\begin{equation}
\label{eq.28}
\lim_{n\to +\infty}\sup_{k\in\mathbb{Z}} \no{u_n}_{L\sp 2 (J_k)} = 0,
\end{equation}
then the sequence of $ L\sp d $ norms of $ (u_n) $
also converges to zero. Here we write a proof which provides
an estimate of the $ L^d $ norm by a product of the $ H^1 $ norm and
\eqref{eq.28}. We need the Sobolev-Gagliardo-Nirenberg inequality for the
interval $ J_k $
\begin{equation}
\label{eq.36}
\no{u}_{L\sp d (J_k)}\sp d\leq
s^d\no{u}_{L\sp 2 (J_k)}\sp{\frac{d + 2}{2}}
\no{u}_{H\sp 1 (J_k)}\sp{\frac{d - 2}{2}}.
\end{equation}
\begin{lemma}
\label{lem:vanishing}
For every $ u\in\X $ there holds
\begin{equation}
\label{eq.v1}
\no{u}_{L\sp d}\sp d\leq s^d
\Big(\sup_{k\in\mathbb{Z}} \no{u}_{L\sp 2 (J_k)}\Big)\sp{d - 2}
\no{u}_{H\sp 1}\sp 2
\end{equation}
if $ d\geq 6 $ and
\begin{equation}
\label{eq.v2}
\no{u}_{L\sp d}\sp d \leq s\sp d
\Big(\sup_{k\in\mathbb{Z}} \no{u}_{L\sp 2 (J_k)}\Big)\sp{\frac{d + 2}{2}}
\no{u}_{H\sp 1}\sp{\frac{d - 2}{2}}
\end{equation}
if $ d\leq 6 $.
\end{lemma}
\begin{proof}
First case: $ d\geq 6 $, that is $ (d - 2)/2 \geq 2 $.
\[
\no{u}_{H\sp 1 (J_k)}\sp{\frac{d - 2}{2}} =
\no{u}_{H\sp 1 (J_k)}\sp 2 \no{u}_{H\sp 1 (J_k)}\sp{\frac{d - 6}{2}}.
\]
Then, in \eqref{eq.36} we have the product
of the bounded sequence
\[
a_k := \no{u}_{L\sp 2 (J_k)}\sp{\frac{d + 2}{2}}
\no{u}_{H\sp 1 (J_k)}\sp{\frac{d - 6}{2}},\quad \no{a}_{\infty} =
\Big(\sup_{k\in\mathbb{Z}} \no{u}_{L\sp 2 (J_k)}\Big)\sp{\frac{d + 2}{2}}
\no{u}_{H\sp 1}\sp{\frac{d - 6}{2}}
\]
and the summable sequence
\[
b_k := \no{u}_{H\sp 1 (J_k)}\sp 2,\quad \no{b}_1 = \no{u}_{H\sp 1}\sp 2.
\]
Therefore,
\[
\no{u}_{L\sp d}\sp d\leq
s\sp d
\Big(\sup_{k\in\mathbb{Z}} \no{u}_{L\sp 2 (J_k)}\Big)\sp{\frac{d + 2}{2}}
\no{u}_{H\sp 1}\sp{\frac{d - 2}{2}}.
\]
Second case: $ d\leq 6 $. Then $ (d - 2)/2 \leq 2 $, and we have
\[
\frac{6 - d}{2} = 2 - \frac{d - 2}{2} < \frac{d + 2}{2}.
\]
Then,
\[
\no{u}_{L\sp 2 (J_k)}\sp{\frac{d + 2}{2}}
\no{u}_{H\sp 1 (J_k)}\sp{\frac{d - 2}{2}}\leq
\no{u}_{L\sp 2 (J_k)}\sp{\frac{d + 2}{2} - \frac{6 - d}{2}}
\no{u}_{H\sp 1 (J_k)}\sp{\frac{d - 2}{2} + \frac{6 - d}{2}} =
\no{u}_{L\sp 2 (J_k)}\sp{d - 2}
\no{u}_{H\sp 1 (J_k)}\sp 2.
\]
Again, taking the sum over $ \mathbb{Z} $, we obtain
\[
\no{u}_{L\sp d}\sp d\leq s^d
\Big(\sup_{k\in\mathbb{Z}} \no{u}_{L\sp 2 (J_k)}\Big)\sp{d - 2}
\no{u}_{H\sp 1}\sp 2.
\]
\end{proof}
\begin{lemma}
\label{lem.L2-H1}
Let $ (w_n)\subseteq\X $ be a sequence such that
\[
E(w_n)\ra I(\lambda)\text{ and } M(w_n)\ra\lambda.
\]
Suppose that $ (w_n) $ converges to $ w $ in $ L\sp 2 $.
Then there exists a subsequence of $ (w_n) $ which converges strongly
in $ \X $.
\end{lemma}
\begin{proof}
From \eqref{eq.14}, $ (w_n) $ is bounded in $ H^1 $ (in the inequality
$ \lambda = M(w_n) $). Then, there exists a subsequence $ (w_{n_k}) $
which converges weakly to $ w $ in $ H\sp 1 $, and pointwise a.e.
\[
\begin{split}
o(1) + I(\lambda) &= E(w_{n_k}) =
\frac{1}{2}\nint |w_{n_k}' (x)|\sp 2 dx + \nint F(w_{n_k}(x)) dx \\
&\geq \frac{1}{2}\nint |w'(x)|\sp 2 dx +
\nint F(|w(x)|) dx \geq I(\lambda) + o(1).
\end{split}
\]
Since $ (w_{n_k}) $ converges pointwise a.e. and $ L^2 $, by \eqref{eq.SGN},
$ \int F(w_{n_k}) dx $ converges to $ \int F(w)dx $. From this fact and
the weak lower-semicontinuity of $ \int |u'|^2 dx $, we obtained
the first inequality. The second inequality follows from the strong
convergence in $ L\sp 2 $
which implies that $ w $ is in $ S(\lambda) $.
Then, taking the limit,
\[
\lim_{k\ra +\infty}\nint |w_{n_k}'(x)|\sp 2 dx =
\nint |w'(x)|\sp 2 dx
\]
implying the convergence of $ w_{n_k}' $ to $ w' $ in $ L\sp 2 $.
In the next lemma, $ \lambda_* $ is as in Proposition~\ref{prop.1}.
\end{proof}
\begin{lemma}
\label{lem.concentration-compactness}
Let $ \lambda $ be such that $ \lambda > \lambda_* $. A subsequence of
a sequence $ (u_n) $ such that
\[
M(u_n) \to\lambda,\quad E(u_n)\ra I(\lambda)
\]
satisfies the concentration behaviour \eqref{eq.C}.
\end{lemma}
\begin{proof}
We show that $ (u_n) $ does not vanish and does not have a dichotomy.
If $ (u_n) $ vanishes, from (V), up to extract a subsequence
\begin{equation}
\label{eq.18}
\lim_{n\to +\infty}\sup_{k\in\mathbb{Z}} \no{u_n}_{L\sp 2 (J_k)} = 0.
\end{equation}
Otherwise, there exists $ \var_0 > 0 $ and a sequence $ (k_n) $ such that
\[
\no{u_{k_n}(\cdot -k_n)}_{L\sp 2 ((0,1))} =
\no{u_{k_n}}_{L\sp 2 (J_{k_n})}\geq\var_0
\]
However, a subsequence
of $ u_{k_n} (\cdot + y_k) $ converges weakly to zero and,
since $ (0,1) $ is bounded, the $ L\sp 2 $-norm converges to zero,
giving a contradiction.

From \eqref{prop.1.7}, $ I(\lambda) > 0 $.
From the inequalities \eqref{eq.v1} and \eqref{eq.v2}, and \eqref{eq.18},
a subsequence of $ (u_{k_n}) $ converges
to zero in $ L\sp p \cap L\sp q $. Therefore, $ I(\lambda) = 0 $,
and we have a contradiction.

Let $ (u_{n_k}) $, $ (y_k) $ and $ u $ be as in \eqref{eq.D}.
Firstly, we observe that the inequality
\[
0 < \no{u}_{L\sp 2} < \liminf_{k\ra +\infty} \no{u_{n_k}}_{L\sp 2}
\]
holds too. Otherwise, we had strong convergence in $ L\sp 2 $ and
thus, strong convergence in $ H\sp 1 $, by Lemma~\ref{lem.L2-H1}
and a contradiction with the dichotomy assumption.
We define
\[
\lambda_1 \sp k:= \no{u_{n_k} (\cdot + y_k) - u}_{L\sp 2}\sp 2.
\]
Up to extract a subsequence, we can suppose that $ \lambda_1 \sp k $
converges. We use the notation $ \lambda_1 $ for its limit and we have
\[
0 < \lambda_1 < \lambda.
\]
We set
\[
w_k := \frac{\lambda_1}{\lambda_1\sp k}\cdot (u_{n_k} (\cdot + y_k) - u).
\]
By \eqref{prop1:2} of Proposition~\ref{prop1},
the sequence $ (w_k) $ is bounded. Then, we can apply \eqref{prop1:8}
and \eqref{prop.1.3}
\begin{equation}
\label{eq.31}
E(u_{n_k}) = E(u_{n_k} (\cdot + y_k) - u) + E(u) + o(1) =
E(w_k) + E(u) + o(1).
\end{equation}
Here, we use the argument of \cite[Lemma~20,p.\,5]{BF14}.
We define
\[
\Lambda(u) := \frac{E(u)}{M(u)}.
\]
By (v) of Proposition~\ref{prop1} and \eqref{eq.31} we have
\[
o(1) + \frac{I(\lambda)}{\lambda} =
\frac{E(w_k) + E(u)}{M(w_k) + M(u)} + o(1)
\geq
\min\{\Lambda(w_k),\Lambda(u)\} + o(1)
\]
Let us suppose that the term of $ w_k $ is the smaller (on the
other case, the argument is the same). Then
\[
o(1) + \frac{I(\lambda)}{\lambda}\geq
\frac{I(\lambda_1)}{\lambda_1} + o(1)\geq
\frac{I(\lambda)}{\lambda} + o(1).
\]
The last inequality is a consequence of \eqref{prop1:6} of
Proposition~\ref{prop1}: the function $ I(\lambda)/\lambda $ is
decreasing. Then, all the inequalities are equalities.
\[
\frac{I(\lambda)}{\lambda} = \frac{I(\lambda_1)}{\lambda_1}.
\]
From \eqref{prop1:6} of Proposition~\ref{prop1}, either $ \vartheta := \lambda/\lambda_1 = 1 $, which is ruled out by the dichotomy assumption,
or $ I(\lambda) = 0 $, which contradicts the assumptions on $ \lambda $.
Thus, the sequence is not dichotomy.
\end{proof}
\begin{proposition}
\label{prop.minimum}
$ \GS\neq\emptyset $ for every $ \lambda > \lambda_* $, and
the Lagrange multiplier is negative. If $ \lambda\leq\lambda_* $,
then $ \GS $ is empty.
\end{proposition}
\begin{proof}
The second part is just \eqref{prop.1.7} of Proposition~\ref{prop.1}.
Let $ (u_n) $ be a minimizing sequence. Since $ \lambda > \lambda_* $
the assumptions of Lemma~\ref{lem.concentration-compactness} are
satisfied and there exists $ (y_n)\subseteq\R $ and $ u\in S(\lambda) $
such that
\[
u_n (\cdot + y_n)\to u.
\]
From Proposition~\ref{prop.regular}, $ E $ is continuous. Then, taking the
limit as $ n\to +\infty $ in
\[
o(1) + I(\lambda) = E(u_n) = E(u_n (\cdot + y_n)) = o(1) + E(u)
\]
we obtain $ E(u)\in S(\lambda) $. Now, we do not set any restriction
on $ \lambda $ and just assume that $ u\in\GS $. By (i) of
Proposition~\ref{prop.regular}, there exists $ \omega\in\R $ such that
\[
u''(x) - f(u(x)) - \omega u(x) = 0.
\]
Taking the scalar product in $ \C $ with $ u' $ and obtain
\begin{equation}
\label{eq.24}
u'(x)\sp 2 - 2F(u(x)) - \om u(x)\sp 2\equiv d
\end{equation}
for some constant $ d $. On the left side we have a sum of $ L^1 $
functions. Therefore, $ d = 0 $. Integrating
on $ \R $, we obtain
\[
\nint |u'(x)|\sp 2 dx - 2\nint F(u(x))dx - \om\lambda = 0.
\]
Since $ u $ is a minimum, the equality above becomes
\begin{equation}
\label{thm.existence.1}
2\left(\nint |u'(x)|\sp 2 dx - I(\lambda)\right) = \omega\lambda.
\end{equation}
From \eqref{prop1:5} of Proposition~\ref{prop.1}, the left term
is non-negative. Then $ \omega > 0 $.
\end{proof}
\begin{remark}
\label{rem.critical}
The critical case $ G(s) = as^6 $ has been already ruled out by the
assumption \eqref{G2}. In this case,
a minimum does not exist. On the contrary,
the rescaling
\[
u_\eta = \eta^{1/2} u(\eta x)
\]
gives $ E(u_\eta) = \eta^2 E(u) = \eta^2 I(\lambda) $. Therefore,
$ I(\lambda) = 0 $ unless $ E $ is unbounded from below. By \eqref{prop.1.7}
of Proposition~\ref{prop.1}, a minimum does not exist.
\end{remark}
We conclude this section by showing general properties satisfied
by minima of $ E $ over $ S(\lambda) $.
\begin{lemma}
\label{lem.minima}
Let $ u $ be a minimum of $ E $ over $ S(\lambda) $.
Then $ R(x) := |u(x)| $ satisfies the following properties:
\begin{enumerate}[(i)]
\item
\label{lem.minima.1}
$ \lim_{|x|\to +\infty} R(x) = 0 $
\item
\label{lem.minima.2}
$ R $ is symmetrically decreasing with respect to a point of
$ \RN $
\item
\label{lem.minima.3}
$ R $ is positive
\item there exists $ z $ such that $ |z| = 1 $ and $ u(x) = zR(x) $
for every $ x\in\R $.
\label{lem.minima.4}
\end{enumerate}
\end{lemma}
\begin{proof}
Clearly, $ R $ is in $ S(\lambda) $.
From the equality $ F(s) = F(|s|) $ and the Convex Inequality for the
Gradient,
\cite[Theorem~7.8]{LL01},
there holds $ E(u)\geq E(R) $. Since $ u $ is a minimum,
necessarily
\begin{equation}
\label{eq.21}
E(u) = E(R).
\end{equation}
Thus, $ R $ is solution to \eqref{eq.E} for some $ \om $.
Since $ R $ is $ H\sp 1 $ it is also $ L\sp{\infty} $.
From \eqref{eq.24} and the continuity of $ F $, the function $ |R'| $
is bounded. Since $ R\in L\sp 2 $, we obtain
\begin{equation}
\label{eq.23}
R (x)\ra 0\text{ as } |x|\ra +\infty
\end{equation}
which is the condition \cite[6.1]{BL83a} (following
their notation $ G'(s) $ should be replaced with $ - G'(s) - \omega s $).
By \cite[(ii), Proof~of~Theorem~5]{BL83a}, $ R $ is positive;
by \cite[(i,iv), Proof~of~Theorem~5]{BL83a}, $ R $ is a symmetrically
decreasing function with respect to a point in $ y $ in $ \R $.
Then, $ R' $ converges to zero as well.
So, we proved (\ref{lem.minima.1},\ref{lem.minima.2}) and
\eqref{lem.minima.3}.

(iv). From \eqref{eq.21} and $ F(s) = F(|s|) $,
there also holds
$ \no{u'}_{L\sp 2}\sp 2 = \no{|u|'}_{L\sp 2}\sp 2 $.
Since $ R > 0 $, by \cite[Lemma~5.1]{Gar12}, there exists a
complex number $ z $ such that $ |z| = 1 $ and $ u(x) = z|u(x)| $
for every $ x $.
\end{proof}
\section{Non-degeneracy of the minima on $ H^1 _r (\R) $}
\label{sec.non-degenerate}
In this section we prove the non-degeneracy of the functional
$ E $ when restricted to the sub-manifold $ S(\lambda)\cap H^1 _r (\R) $
on minima.
We need the notation
\begin{equation}
\label{eq.ND24}
Q(\omega,s) := \omega s^2 + 2G(s).
\end{equation}
We have
\begin{equation}
\label{eq.ND25}
R_* (\omega) = \inf\{s > 0\mid Q(\omega,s) = 0\}.
\end{equation}
where $ R_* (\omega) $ has been defined in \eqref{eq.R}.
\begin{remark}
\label{rem.ND1}
If \eqref{G2} holds, then $ R_* $ is a positive non-decreasing
function defined on $ (0,+\infty) $.
\end{remark}
Let $ R_0 $ be an element of $ \GS\cap H^1 _r (\R) $.
Then, there exists $ \omega_0 $ such that
\begin{equation}
\label{eq.ND8}
R_0 ''(x) = G'(R_0 (x)) + \omega_0 R_0 (x).
\end{equation}
By \eqref{lem.minima.1} of Lemma~\ref{lem.minima} $ R_0 $
converges to zero and then satisfies the condition \cite[6.1]{BL83a}.
Then, by (iii) of \cite[Theorem~5]{BL83a},
\begin{equation}
\label{eq.ND26}
R_0 (0) = R_* (\omega_0),\quad\pt_s Q(\omega_0,R_*(\omega_0)) < 0.
\end{equation}
By the Implicit Function Theorem, there exists $ \var_0 > 0 $
such that $ R_* $ is continuously differentiable on
$ (\omega_0 - \var_0,\omega_0 + \var_0) $ and
\[
\pt_s Q(\omega,R_*(\omega)) < 0.
\]
Also, since $ \omega_0 > 0 $, by Proposition~\ref{prop.minimum} and
\eqref{eq.21}, on this
interval $ \omega > 0 $.
We consider the solution of the initial value problem
\begin{equation}
\label{eq.IVP}
R_\omega ''(x) = G'(R_\omega (x)) + \omega R_\omega(x),
\quad R_\omega '(0) = 0,\quad R_\omega(0) = R_* (\omega).
\end{equation}
From \cite[Theorem~5]{BL83a}, $ R_\omega $ converges to
zero and, by \cite[Remark~6.3]{BL83a} and the fact that
$ \omega > 0 $, we obtain $ R_\omega\in H^1 $.
Since $ Q(\omega,R_* (\omega)) = 0 $,
differentiating with respect to $ \omega $, we obtain
\begin{equation}
  \label{eq.NDS9k}
  (2G'(R_*(\omega)) + 2\omega R_*(\omega)) R_* '(\omega) +
  R_* ^2(\omega)= 0.
\end{equation}
We define
\[
\lambda\colon (\omega_0 - \var_0,\omega_0 + \var_0)\to\R,\quad
\lambda(\omega) := \no{R_\omega}_{L^2} ^2.
\]
\begin{equation}
\label{eq.NDS9a}
\omega R_* (\omega)^2 + 2G(R_* (\omega)) = 0.
\end{equation}
\begin{lemma}
\label{lem.non-degeneracy}
$ \lambda'(\omega)\geq 0 $ and $ \lambda'(\omega_0) > 0 $,
provided \eqref{G4} holds.
\end{lemma}
\begin{proof}
From (iv) of \cite[Theorem~5]{BL83a}, $ R_\omega $ is a strictly
decreasing function on $ |x| $. Then, since $ R_\omega $ is
real valued, from \eqref{eq.24} we have
\[
R_\omega '(x)^2 = \omega R_\omega (x) ^2 + 2G(R_\omega(x))
\]
and then
\[
R_\omega '(x) = -\sqrt{\omega R_\omega (x)^2 + 2G(R_\omega (x))}.
\]
We can write
\[
\begin{split}
\lambda &= 2 \int_0^\infty  R_\omega (x)^2 dx = 2 \int_0^\infty
\frac{R_\omega (x)^2}{-\sqrt{\omega R_\omega (x)^2 + 2G(R_\omega(x))}}
R_\omega ^\prime (x) dx \\
&= 2 \int_{0}^{R_*(\omega)} \frac{\rho^2 d\rho}%
{\sqrt{\omega \rho^2 + 2G(\rho)}}
= 2 \int_{0}^{1}
\frac{\theta^2 d\theta}{\sqrt{\Psi(\theta,R_*(\omega),\omega)}}.
\end{split}
\]
The third and the fourth equalities follow from the substitutions
$ \rho = R(x) $ and $ \rho = R_* (\omega)\theta $, and
\[
\Psi(\theta,s,\omega) = \omega \theta^2 s^{-4} + 2 s^{-6}G(s\theta).
\]
We prove that the function
\[
\omega \to \Psi(\theta, R_*(\omega), \omega)
\]
is non-increasing in $ \omega $. Then, we have to check that
\begin{equation}\label{eq.mp3}
\partial_\omega \Psi(\theta, R_*(\omega), \omega)\leq 0.
\end{equation}
In turn, the derivative above is equal to
\[
\begin{split}
(R_*(\omega))^{-4}\theta^2
&+ \left[-4\omega\theta^2(R_*(\omega))^{-5} -
12 (R_*(\omega))^{-7} G(R_*(\omega)\theta)\right.\\
&+ \left. 2 (R_*(\omega))^{-6} \theta
G'(R_*(\omega)\theta) \right] R_*^\prime(\omega).
\end{split}
\]
From Remark~\ref{rem.ND1}, the term
\[
R_* (\omega)^7 (R_* ' (\omega))^{-1}
\]
is positive. Then, dividing $ \partial_\omega\Psi $ by that term
and using the relation \eqref{eq.NDS9k}, we obtain
\[
\begin{split}
I(\omega,\theta) = &-R_*(\omega)\theta^2\left[ 2 \omega  R_*(\omega)
- 2G'(R_*(\omega)\right] - 4\omega \theta^2 (R_*(\omega))^2\\
&- 12 G(R_*(\omega)\theta) + 2 R_*(\omega) \theta G'(R_*(\omega)\theta)
\end{split}
\]
so using \eqref{eq.NDS9a} we see that
\[
\begin{split}
I(\omega,\theta) &= 12 \theta^2 G(R_*(\omega)) -
2\theta^2 R_*(\omega) G'(R_*(\omega))  \\
&+ 12 G(R_*(\omega)\theta) -
2R_*(\omega)\theta G'(R_*(\omega)\theta).
\end{split}
\]
Setting
\[
H(s) := - 6G(s) + sG'(s),
\]
we obtain
\[
\begin{split}
I(\omega,\theta) &= 2H(R_*(\omega)\theta) - 2\theta^2 H(R_*(\omega)) \\
&= 2\theta^2 R_* (\omega) ^2\left(\frac{H(R_*(\omega)\theta)}%
{R_* (\omega)^2 \theta^2} -
\frac{H(R_* (\omega)}{R_* (\omega)^2}\right).
\end{split}
\]
Now we prove that the function $ H(s)/s^2 $ is monotonically non-decreasing
on the interval $ (0,R_* (\omega)) $.
Equivalently, we need to check that
\[
H' s - 2 H = 12G(s) - 7sG'(s) + s^2 G''(s)\geq 0.
\]
If we require \eqref{G3}, the inequality holds.
Moreover, $ I(\omega,1) = 0 $.
Then, for every $ 0\leq\theta\leq 1 $, we have $ I(\omega,\theta)\leq 0 $.
In conclusion,
\[
\frac{d\lambda}{d\omega} = -\int_0\sp 1
\frac{\theta^2 \pt_\omega (\Psi(\theta,R_* (\omega),\omega)) d\theta}%
{(\Psi(\theta,R_*(\omega),\omega))^{3/2}}\geq 0.
\]
We are now able to prove that $ \lambda'(\omega_0) > 0 $.
On the contrary,
\[
\partial_\omega (\Psi(\theta,R_*(\omega_0),\omega_0)) = 0
\]
for every $ 0 < \theta < 1 $, and the same applies to $ I $.
Therefore,
\[
12G(s) - 7sG'(s) + s^2 G''(s) = 0\text{ for every } s\in (0,R_* (\omega_0)).
\]
Then $ G(s) = as^6 $ on $ (0,R_* (\omega_0)) $.
By \cite[Theorem~5]{BL83a}, there is only one solution to \eqref{eq.ND8}
which is positive and converges to zero at infinity. Then
$ R_{\omega_0}(0) = R_0(0) $, so the image of $ R_0 $ is contained
in a set of $ \R $ where $ G $ is a pure-power critical non-linearity.
From Remark~\ref{rem.critical}, $ \GS $ is empty, giving a contradiction.
\end{proof}
We wish to evaluate the Hessian operator of $ E $ at the critical point
$ R_0 $, in a vector of the tangent space of $ R_0 $
\begin{equation}
\label{eq.ND7}
T_{R_0} (S_r (\lambda)) = \{v\in H^1 _r (\R)\mid (v,R)_{L^2} = 0\}.
\end{equation}
We consider a curve in $ S_r (\lambda) $ as in
\begin{equation}
  \label{eq.ND2}
  u(t) = R + t v + \alpha(t)R.
\end{equation}
By the Implicit Function Theorem, there exists $ \delta > 0 $
and $ \alpha\colon (-\delta,\delta)\to\R $ such that
\[
M(R + t v + \alpha(t)R) = \lambda.
\]
From the Taylor expansion of $ M $ we get
\[
\alpha(t) = \alpha_0 t^2 + o(t^2),\quad
\alpha_0 = - \frac{\| v \|^2_{L^2}}{2\lambda},
\]
and from the expansion of $ E(u(t)) $ we get
\[
2E(u(t)) = \|R_0 ^\prime\|_{L^2}^2 + 2 t (R_0 ^\prime,
v^\prime)_{L^2} + t^2 \left(\|v^\prime \|^2 + 2\alpha_0
(R_0 ^\prime, v^\prime)_{L^2} \right) + o(t^2)
\]
so using \eqref{eq.E} and \eqref{eq.ND7} we find
\[
2 E(u(t)) = 2E(R_0)
t^2 \left(\|v^\prime \|^2 + ( (G''(R_0)+ \omega R_0) v,v)_{L^2}
\right) + o(t^2).
\]
Therefore,
\[
\begin{split}
D^2 E(R_0)[v,v] &= (E\circ u)''(0) \\
&= 
\nint\big(|v'(x)|^2 + (G''(R_0 (x)) + \omega_0) v(x)^2\big)dx =:
\xi(v).
\end{split}
\]
In order to show that $ R_0 $ is non-degenerate, we have to
prove that the infimum of $ \xi $ is positive
$ T_{R_0} (S_r (\lambda))\cap S(1) $.
If $ v $ is $ H^2 (\R) $, then
\[
\xi(v) = (L_+ v,v)_{L^2},\quad
L_+ (v) := v'' - G''(R_0) v - \omega_0 v.
\]
\begin{proof}[Proof~of~Theorem~\ref{thm.non-degeneracy}]
Since $ R_0 $ is a minimum, $ \xi(v)\geq 0 $.
The infimum of $ \xi $ achieved. A proof of this can be found in
\cite[Proposition~2.10]{Wei85}.
Suppose that the infimum is achieved and that $ \xi(v) = 0 $.
Then, $ v $ is $ H^2 $ and satisfies
\[
L_+ (v) = \beta R_0
\]
for some $ \beta\in\R $ with $ \beta\neq 0 $.
Taking the derivative with respect to $ \omega $ of \eqref{eq.IVP},
and evaluating at $ \omega = \omega_0 $, we obtain
\begin{equation}
\label{eq.ND27}
L_+ (\pt_\omega R(\omega_0,\cdot)) = R_0.
\end{equation}
Then $ y := \beta v_0 + v $ solves the differential equation
$ L_+ (y) = 0 $.
The kernel of the operator $ L_+ $ is generated by $ R_0 ' $,
which is an odd function. Since $ y $ is even, we obtain $ y = 0 $.
Since $ \beta\neq 0 $,
\[
(L_+ (\pt_\omega R(\omega_0,\cdot)),
\pt_\omega R(\omega_0,\cdot))_{L^2} = 0
\]
However, from \eqref{eq.ND27} and the definition $ \lambda $
given in Lemma~\ref{lem.non-degeneracy}, we have
\[
0 = (R_0,\pt_\omega R(\omega_0,\cdot))_{L^2} = \lambda'(\omega_0)
\]
which gives a contradiction with the lemma.
\end{proof}
\def\proofname{Proof}
\begin{corollary}
The set $ \GS\cap H^1 _r (\R) $ is finite.
\end{corollary}
\begin{proof}
Let $ (R_n)\subseteq\GS\cap H\sp 1 _r (\R;\R) $ be a sequence
of minima. By Lemma~\ref{lem.concentration-compactness}, up
to extract a subsequence, we can suppose that
$ R_n (\cdot + y_n) $ converges in $ H\sp 1 $, for
some sequence $ (y_n)\subseteq\R $.
By \cite[Theorem~5]{BL83a}, $ R_n $ is symmetric
and radially decreasing with respect to the origin.
Therefore,
\[
\no{R_n - R_m}_{L\sp 2}\sp 2\leq
\no{R_n (\cdot + y_n) - R_m (\cdot + y_m)}_{L\sp 2}\sp 2.
\]
Then $ (R_n) $ is a Cauchy sequence and there exists $ R_0 $ such that
$ R_n\to R_0 $ in $ L^2 $. By Lemma~\ref{lem.L2-H1}, the convergence
is strong in $ H^1 $, which contradicts the fact that
$ R_0 $ is non-degenerate, thus isolated, minimum.
\end{proof}
\begin{proposition}
\label{prop.finite}
As $ u $ varies in $ \GS $, there are finitely many $ \GU $.
\end{proposition}
\begin{proof}
By \eqref{lem.minima.2} and \eqref{lem.minima.4} of
Lemma~\ref{lem.minima}, there
exists $ y\in\R $ and a complex number $ |z| = 1 $ such that
$ u(x) = zR(x + y) $, where $ R\in\GS\cap H_{r,+} ^1 (\R) $.
Therefore, there are as many different $ \GU $ as
$ \#\GS\cap H_{r,+} ^1 (\R) $.
\end{proof}
\section{Stability of $ \GS $ and $ \GU $}
\label{sec.stability}
According to \cite[Theorem~3.5.1,\,p.\,77]{Caz03} the
The equation \eqref{eq.NLS} is locally well posed in
$ H\sp 1 (\R;\C) $. That is, given $ u $
in $ \X $, there exists a map
\begin{equation}
\label{eq.32}
U\in C\sp 1 \big([0,+\infty);L\sp 2 (\R;\C)\big)\cap
C\big([0,+\infty);H\sp 1 (\R;\C)\big)
\end{equation}
such that $ U_0 = u $ and $ \phi(t,\cdot) := U_t (u) $
is a solution to \eqref{eq.NLS}. We briefly check that $ G' $ satisfies the
the condition of \cite[Example~3.2.4,\,p.\,59]{Caz03}: since
$ G'' $ is continuous,
\[
|G'(s_1) - G'(s_2)|\leq L(K)|s_1 - s_2|,\quad L(K) := \sup_{[0,K]} |G''|
\]
for every $ s_1,s_2\in [0,K] $. And the function $ L $ is continuous
because $ G'' $ is continuous. The global well-posedness
follows from the apriori estimates that one can derive from
\eqref{prop1:2} of Proposition~\ref{prop1}.
\begin{proof}[Proof of Theorem~\ref{thm.ground-state}]
The proof of the stability is made with a contradiction argument:
let $ (u_n) $, $ \var_0 > 0 $ and $ (t_n) $ be such that
\begin{equation}
\label{eq.contradiction}
\dist(u_n,\GS)\to 0,\quad\dist(U_{t_n} (u_n),\GS)\geq\var_0.
\end{equation}
Since $ E $ and $ M $ are continuous functions, and constant
on the orbits $ U_t (u_n) $,
\[
E(U_{t_n}(u_n))\to I(\lambda),\quad M(u_n)\to\lambda.
\]
We set $ v_n := U_{t_n} (u_n) $.
From Lemma~\ref{lem.concentration-compactness}, there exists a
subsequence $ u_{n_k} $, a sequence $ (y_k) $ and
$ u\in S(\lambda) $ such that
\[
\no{u_{n_k} (\cdot + y_k) - u}_{H\sp 1 (\R;\C)}\to 0
\]
implying
\[
\no{u_{n_k} - u(\cdot - y_k)}_{H\sp 1 (\R;\C)}\to 0
\]
and giving a contradiction with \eqref{eq.contradiction}.
\end{proof}
\begin{proof}[Proof~of~Theorem~\ref{thm.sw}]
\textsl{Stability of $ \GU $}. By Proposition~\ref{prop.finite},
\[
\GS = \bigcup_{i = 1}\sp n \mathcal{G}_\lambda (R_i).
\]
We prove that
\[
\dist(\GS(R_h),\GS(R_k)) > 0\text{ for } h\neq k.
\]
In fact, the distance between two arbitrary points in the two
sets is
\[
\begin{split}
\dist(z_1 R_h (\cdot + y_1),z_2 R_k(\cdot + y_2)) &=
\dist(R_h, zR_k (\cdot + y))\\
&\geq\dist(R_h,R_k) > 0,
\end{split}
\]
where $ z =  \overline{z}_1 z_2 $ and $ y := y_2 - y_1 $.
The first inequality follows the fact that both $ R_i $ are symmetrically
decreasing with respect to the origin, from \eqref{lem.minima.2} of
Lemma~\ref{lem.minima}. Then
\begin{equation}
\label{eq.distance}
d := \inf_{h\neq k} \dist(\GS(R_h),\GS(R_k)) > 0.
\end{equation}
Now we prove that $ \mathcal{G}_{\lambda} (R_i) $ is stable.
Let $ \delta > 0 $ be such that
\begin{equation}
\label{eq.isolated}
B(R_i,\delta)\cap\GS\cap H^1 _r = \{R_i\},\quad \delta < \frac{d}{3}.
\end{equation}
We define
\[
E_\delta\sp i := \inf_{B(\GS(R_i),\delta)} E
\]
where the metric restricted on $ S(\lambda) $.
We claim that
\begin{equation}
\label{eq.11}
E_\delta\sp i > I(\lambda).
\end{equation}
Otherwise, we would have a sequence $ (u_n) $ such that
\[
E(u_n)\to I(\lambda),\quad M(u_n) = \lambda.
\]
By Lemma~\ref{lem.concentration-compactness}, there exists
a subsequence $ (u_{n_k}) $, $ u $ in $ S(\lambda) $ and $ (y_k) $ such that
\begin{equation}
\label{eq.17}
u_{n_k} (\cdot + y_k)\ra u\in\GS.
\end{equation}
From \eqref{eq.isolated} and the choice of $ \delta $, it follows
that $ u $ is in $ \GS(R_i) $. However, since
\[
\dist(u_{n_k} (\cdot + y_k),\GS(R_i)) = \delta
\]
there also hold $ \dist(u,\GS(R_i)) = \delta $, giving a
contradiction with \eqref{eq.17}. We are now able to
prove that $ \GS(R_i) $ is stable; again,
we use a contradiction argument. Let $ (u_n) $,
$ (t_n) $ and $ \var_0 > 0 $ be such that
\[
\dist(u_n,\GS(R_i))\to 0,\quad
\dist(U_{t_n} (u_n),\GS(R_i))\geq\var_0.
\]
We set $ v_n := U_{t_n} (u_n) $. Since $ \GS $ is stable, there exists $ k $
such that
\[
\dist(v_n,\GS(R_k))\to 0,\quad
\GS(R_i)\neq\GS(R_k).
\]
Let $ n_0 $ be such that
\begin{equation}
\label{eq.10}
\max\{\dist(u_n,\GS(R_i)),\dist(v_n,\GS(R_k)\}
< \delta
\end{equation}
and
\begin{equation}
\label{eq.12}
E(u_n) < \min\{E_\delta\sp i\mid 1\leq i\}
\end{equation}
for every $ n\geq n_0 $.
Along the curve
\[
\alpha\colon [0,t_{n_0}]\ra H\sp 1 (\R;\C),\quad
\alpha(t) = U_t (u_{n_0})
\]
the quantities $ E $ and $ M $ are constant, while the function
\[
\beta\colon\R\ra\R,\quad \beta(t) := \dist(U_t (u_{n_0}),\GS(R_i))
\]
is continuous, from \eqref{eq.32}. From \eqref{eq.10} and \eqref{eq.isolated},
we have
\[
\beta(0) < \delta,\quad \beta(t_{n_0})\geq 2\delta.
\]
Therefore, there exists $ t_* $ such that
\[
\beta(t_*) = \delta.
\]
Then,
\[
E(\alpha(t_*))\geq E_\delta\sp i.
\]
However from the conservation of $ E $ and \eqref{eq.12}
\[
E(\alpha(t_*)) = E(\alpha(0)) < E_\delta\sp i.
\]
And from \eqref{eq.11}, we obtain a contradiction.
\end{proof}
\begin{TAKEOUT}
Now, we wish to prove that solitary-waves are stable, in the following
sense: given $ v $ in $ \R $, the set
\[
G_\lambda (u,v) :=
\big\{ze\sp{ixv} u(\cdot + y)\mid (z,y)\in S\sp 1\times\R\big\}.
\]
is stable. The results relies on \eqref{eq.33}, the fact that
$ G_\lambda (u,0) = G_\lambda (u) $ is stable, and the next lemma.
We define $ S(x) := e\sp{ixv} $. From the local uniqueness of solutions
to \eqref{eq.NLS}, if $ \phi $ is a solution to \eqref{eq.NLS}, then
\[
\psi(t,x) := e\sp{i(v\cdot x - t|v|\sp 2)}\phi(t,x - 2tv)
\]
is a solution to \eqref{eq.NLS}.
Therefore,
\begin{equation}
\label{eq.33}
U_t (S\Psi)(x) = S(x) e\sp{-it|v|\sp 2} U_t (\Psi)(x - 2tv).
\end{equation}
It is convenient to define the following linear operators
\[
L_v \colon\X\ra\X,\quad L(\Phi) = e\sp{ixv}\Phi,\quad
B_t (\Phi) := e\sp{-it|v|\sp 2} \Phi(x - 2tv).
\]
\begin{lemma}
\label{lem:semigroups}
For every $ v,w,t $ in $ \R $
\begin{enumerate}[(i)]
\item
\label{lem:semigroups-1}
$ L_{v + w} = L_v \circ L_w $
\item
\label{lem:semigroups-2}
$ L_v $ is bounded and
$ \no{L_v}_{\mathcal{B}(\X)}\leq\sqrt{2(1 + |v|\sp 2)} $
\item
\label{lem:semigroups-3}
$ L_v (G_\lambda (u,w)) = G_\lambda (u,v + w) $
\item
\label{lem:semigroups-4}
$ B_t $ is an isometry of $ \X $
\item
\label{lem:semigroups-5}
$ B_t (G_\lambda (u,v)) = G_\lambda (u,v) $.
\end{enumerate}
\end{lemma}
\begin{proof}
We check only \eqref{lem:semigroups-2}, where some computation is
required, while all the other statements can be safely checked by
the reader:
\begin{align*}
\no{L_v (\Phi)}_{L\sp 2}\sp 2 &=
\text{Re}\nint e\sp{ixv} \Phi(x) e\sp{-ixv} \overline{\Phi}(x)dx =
\no{\Phi}_{L\sp 2}\sp 2\leq 2\no{\Phi}_{H\sp 1}\sp 2\\
\no{L_v (\Phi)'}_{L\sp 2}\sp 2 &=
\no{e\sp{ixv} (\Phi + iv\Phi')}_{L\sp 2}\sp 2\leq
2(\no{\Phi}_{L\sp 2}\sp 2 + |v|\sp 2 \no{\Phi'}_{L\sp 2}\sp 2)\leq 2|v|\sp 2
\no{\Phi}_{H\sp 1}\sp 2.
\end{align*}
Taking the sum, we obtain the desired estimate.
\end{proof}
It is useful to write the relation in \eqref{eq.33} using the operators
$ B_t $ and $ U $:
\begin{equation}
\label{eq.35}
L_v\sp {-1}\circ U_t = B_t\circ U_t\circ L_v\sp{-1}.
\end{equation}
\begin{proof}[Proof~of~Corollary~\ref{cor:1}]
In order to simplify the notation we set
\[
G := G_\lambda (u),\quad G_* := G_\lambda (u,v),\quad L := L_v.
\]
Then
\[
\begin{split}
\dist(U_t (\Psi),G_*) &=
\dist(U_t (\Psi),L(G))\leq\sqrt{2(1 + |v|\sp 2)}
\dist(L\sp{-1} U_t (\Psi),G) \\
&=
\sqrt{2(1 + |v|\sp 2)}\dist(B_t (U_t (L\sp{-1} (\Psi))),G) \\
&=
\sqrt{2(1 + |v|\sp 2)}\dist(U_t (L\sp{-1} (\Psi),G)).
\end{split}
\]
Since $ G $ is stable, given $ \var > 0 $, there exists $ \delta > 0 $
such that
\[
\dist(\Phi,G) < \delta\implies\dist(U_t(\Phi),G) <
\frac{\var}{\sqrt{2(1 + |v|\sp 2)}}.
\]
We choose
\[
\delta_v := \frac{\delta}{\sqrt{2(1 + |v|\sp 2)}}
\]
If $ \dist(\Psi,G_*) < \delta_v $, then
$ \dist(L\sp{-1}(\Psi),G) < \delta $. Then
\[
\dist(L\sp{-1}(\Psi),G) < \frac{\var}{\sqrt{2(1 + |v|\sp 2)}}.
\]
From the chain of inequalities in \eqref{eq.}, we obtain
\[
\dist(U_t(\Psi),G_*) < \var.
\]
\end{proof}
\begin{proof}[Proof of Corollary~\ref{cor:2}]
Let $ M $ be a stable and invariant subset of $ G_\lambda $,
that is $ U_t (M)\subseteq M $ for every $ t\geq 0 $.
We show that $ M $ is closed under multiplication by complex numbers in
$ S\sp 1 $ and space translations. That is, given $ \Phi $ in $ M $,
there holds
\[
z\Phi(\cdot + y)\in M
\]
for every $ (z,y)\in S\sp 1\times\R $. We set
$ u(x) := |\Phi(x)|\in G_\lambda $. From Lemma~\ref{lem:minima},
there exists $ z $ in $ S\sp 1 $ such that $ \Phi(x) = zu(x) $.
There also exists $ \om > 0 $ such that
\[
U_t (\Phi)(x) = ze\sp{i\om t} u(x) = e\sp{i\om t}\Phi(x).
\]
Since $ M $ is invariant, for every $ t\geq 0 $, the function above
belongs to $ M $. Since $ \om > 0 $, therefore non-zero,
the image of $ e\sp{i\om t} $
is $ S\sp 1 $.
Now, we prove that $ M $ is invariant under space translations.
Let $ y $ be a point of $ \R $.
We will show that $ u(\cdot + y) $ is in $ M $.
We consider the example of T.~Cazenave
and P.~L.~Lions, \cite{CL82} and define
\[
\Phi_n (x) := e\sp{\frac{ix}{n}} u(x).
\]
There holds
\[
\dist(\Phi_n,M)\ra 0.
\]
From \eqref{eq.33}, we have
\[
U_t (\Phi_n)(x) = e\sp{i\left(\frac{nx - t}{n\sp 2}\right)} u(x - 2t/n).
\]
We set $ t_n := -yn/2 $ and obtain
\[
U_{t_n}(\Phi_n)(x) = e\sp{-iy/2n} e\sp{ix/n} u(x + y).
\]
Since $ M $ is invariant for the action of $ S\sp 1 $,
\[
\dist(U_{t_n}(\Phi_n),M) = \dist(e\sp{ix/n} u(\cdot + y),M)
\]
Since
\[
\dist(e\sp{ix/n} u(\cdot + y),u(\cdot + y))\ra 0
\]
we have
\[
\dist(e\sp{ix/n} u(\cdot + y),M) =
\dist(u(\cdot + y),M) + o(1).
\]
Since $ M $ is stable and closed, $ u(\cdot + y)\in M $.
Therefore, whenever $ u_i\in M $, we have $ G_\lambda (u_i)\subseteq M $.
Then
\[
M = \bigcup_{u\in M} G_\lambda (u).
\]
In particular, among the closed, invariant and stable subsets of
$ G_\lambda $, $ G_\lambda (u) $ is minimal.
\end{proof}
\end{TAKEOUT}
\section{Uniqueness}
\label{sec.uniqueness}
We assume that (G1-5) hold. We fix $ \lambda > 0 $.
\def\proofname{Proof~of~Theorem~\ref{thm.uniqueness}}
\begin{proof}
Let $ R_0 $ and $ R_1 $ be two positive functions in $ \GS\cap H^1 _r $.
The set $ \mathscr{A} $ introduced in \eqref{G5} is the set of $ \omega $
such that a solution to \eqref{eq.E} exists. If $ \mathscr{A} $ is
connected, then
the function $ R_* $ defined on $ (\omega_0 - \var_0,\omega_0 + \var_0) $,
in \eqref{eq.ND26}, can be extended to $ \mathscr{A} $, so the function
$ \lambda $.
Let $ \omega_1 $ be the Lagrange multiplier associated to $ R_1 $.
Then $ \omega_1\in A $. Since $ R_0 $ and $ R_1 $ belong to the same
constraint,
\[
\lambda(\omega_0) = \lambda(\omega_1).
\]
By Lemma~\ref{lem.non-degeneracy}, $ \lambda'\geq 0 $ on $ [\omega_0,\omega_1] $.
Then $ \lambda'(\omega) = 0 $ on the whole interval. Then
$ \lambda'(\omega_0) = 0 $ giving a contradiction with
Lemma~\ref{lem.non-degeneracy}. Hence $ \omega_0 = \omega_1 $ and $ R_0 $
and $ R_1 $ solve the same initial value problem \eqref{eq.IVP}. Then
$ R_0 = R_1 =: R_+ $. The other solution is $ R_- := -R_+ $.
\end{proof}
\def\proofname{Proof}
\begin{corollary}
\label{cor.uniqueness}
If (G1-5) hold, then $ \GS = \GU $ for every $ u\in\GS $.
\end{corollary}
\begin{proof}
We prove that an arbitrary $ v\in\GS $ belongs to $ \GS $. In fact,
by \eqref{lem.minima.4} of Lemma~\ref{lem.minima},
there are two complex numbers $ z,w\in\C $ such that
$ |z| = |w| = 1 $ and
\[
v(x) = zR_1 (x),\quad u(x) = w R_2 (x)
\]
where $ R_1,R_2\in\GS\cap H^1 $ and symmetric with respect to two points
$ y_1 $ and $ y_2 $, respectively, by \eqref{lem.minima.2} of
Lemma~\ref{lem.minima}. Then
$ R_1 (\cdot - y_1) $ and $ R_2 (\cdot - y_2) $ are two positive solutions
in $ \GS\cap H^1 _r $. By Theorem~\ref{thm.uniqueness},
\[
R_1 (\cdot - y_1) = R_2 (\cdot - y_2).
\]
Then
\[
\begin{split}
v(x) &= zR_1 (x) = zR_2 (x - y_2 + y_1) = w^{-1} z w R_2 (x - y_2 + y_1) \\
&= w^{-1} z u(x - y)
\end{split}
\]
where $ y := y_1 - y_2 $. Then $ v\in\GU $.
\end{proof}
\section{The combined power-type case}
\label{sec.cp}
An example of non-linearity $ G $ satisfying all the assumptions
(G1-G5) is
\[
G(s) := - a |s|^p + b |s|^q,\quad 2 < p < 6,\quad 2 < q
\]
with $ c_1,c_2 > 0 $. Regularity and power-type estimate assumptions contained
in \eqref{G1}, \eqref{G2} and \eqref{G4}.

\eqref{G3} is satisfied.
Let $ s_0 > 0 $ be the zero of the function $ Q_\omega $ such that
$ Q_\omega (s_0) = 0 $ and $ Q_\omega '(s_0) < 0 $.
First, we prove that $ L(s_0)\geq 0 $. In fact,
the two conditions on $ s_0 $ give
\[
\omega - c_1 s_0^{p - 2} + c_2 s_0 ^{q - 2} = 0,\quad
2\omega - p c_1 s_0^{p - 2} + q c_2 s_0 ^{q - 2} < 0.
\]
A substitution yields
\[
c_1 (p - 2) s_0 ^{p - 2} - c_2 (q - 2) s_0 ^{q - 2} > 0.
\]
Then
\[
\begin{split}
L(s_0) &= c_1 (p - 2) (6 - p) s_0 ^p - c_2 (q - 2)(6 - q) s_0 ^q \\
&> c_2 (q - 2)(6 - p) s_0 ^q - c_2 (q - 2)(6 - q) s_0 ^q \\
&= c_2 (q - 2) (q - p) s_0 ^{q - 2} > 0.
\end{split}
\]
We can show that $ L $ is non-negative on the interval
$ (0,R_* (\omega)] $. Let $ s < R_* (\omega) $ be such that
$ L(s) < 0 $. Then
$ L(R_* (\omega)) < 0 $, because $ L $ has only one zero on $ (0,+\infty) $.
By definition of $ R_* (\omega) $, we have $ Q_\omega (R_*(\omega)) = 0 $ and
$ Q_\omega ' (R_* (\omega)) < 0 $,
which implies $ L(R_* (\omega)) > 0 $ and gives a contradiction.

\eqref{G5} is satisfied. Let $ s_1 $ be the unique local maximum of $ G $.
Then $ A = H((0,s_1)) $, thus connected. Therefore, from
Theorem~\ref{thm.uniqueness}, when $ G $ is a combined power pure-power
non-linearity, there exists only one positive function
$ R\in\GS\cap H^1 _r $.
\section*{Appendix}
We show that a function satisfied a combined power-type estimate
can be written as sum of two functions satisfying a power-type
estimate. As a consequence, we can suppose that $ G $ satisfies
\begin{equation}
\label{eq.g2}
|G'(s)|\leq c|s|\sp{p - 1},\quad G(0) = 0
\end{equation}
in place of \eqref{G2}, and
\begin{equation}
\label{eq.g4}
|G''(s)|\leq c|s|\sp{p - 2}
\end{equation}
in place of \eqref{G4}.
\begin{proposition}
\label{prop.split}
Let $ G $ be a function satisfying \eqref{G2}.
Then, there are two functions $ C^1 $ functions $ G_1 $ and $ G_2 $
and $ c\in\R $ such that
\[
|G_1 '(s)|\leq c |s|\sp{p - 1},\quad |G_2 '(s)|\leq c|s|\sp{q - 1},\quad
G = G_1 + G_2.
\]
If $ G $ satisfies \eqref{G4} as well, then $ G_1 $, $ G_2 $ and $ c $
can be chosen in such a way that the inequalities
\[
|G_1 '' (s)|\leq c |s|\sp{p - 2},\quad |G_2 ''(s)|\leq c|s|\sp{q - 2}.
\]
are also satisfied.
\end{proposition}
\begin{proof}
In both cases, the function can be obtained as follows: we consider a
non-negative continuous function $ \sigma $ such that
\[
\sigma =
\begin{cases}
1 & \text{on } [-1,1]\\
0 & \text{on } (-\infty,-2]\cup [2,+\infty)
\end{cases}
\]
and $ 0\leq\sigma\leq 1 $, $ |\sigma'|\leq 2 $. Then we choose
$ G_1 $ and $ G_2 $ as the unique functions such that
$ G_1 ' = \sigma G' $, $ G_2 ' = (1 - \sigma) G' $ and
$ c = 2\sp{q - p + 1} C $. If \eqref{G4} holds,
$ G_1 '' = \sigma G'' $ and $ G_2 '' = (1 - \sigma) G'' $, while $ c $
does not change.
\end{proof}
The next proposition is about the regularity of $ E $.
The gradient part of $ E $ is smooth; the regularity
of the non-linear part it is obtained with the same techniques
used by A.~Ambrosetti and G.~Prodi in \cite[Theorem~2.2]{AP93}.
We include the details of this proof in view of slight differences
with the quoted reference, where $ \R $ is replaced
by a bounded domain $ \Omega $, and the class of regularity
$ C^2(\X,\R) $ is replaced by $ C(L^p (\Omega),L^q (\Omega)) $.

The regularity of $ E $ depends on the regularity of $ F $ and the
power-type estimates which, in turn, are related to the estimates
of $ G $ \eqref{eq.g2} and \eqref{eq.g4}.
We identify $ \C $ with $ \R\sp 2 $.
If $ G $ is derivable, then
\begin{gather*}
F(s) = G(|s|)\\
F'(s) = G'(|s|)\frac{s}{|s|},\quad F'(0) = 0
\end{gather*}
for every $ s\in\C - \{0\} $. If $ G $ is two-times derivable,
\begin{gather*}
\pt_{ij} ^2 F(s) = G''(|s|)\frac{s_i s_j}{|s|^2} + G'(|s|)
\frac{\delta_{ij} |s|^2 - s_i s_j}{|s|^3},\quad
\pt^2 _{ij} F(0) = 0
\end{gather*}
Moreover,
\[
\begin{split}
(\pt_{ij} ^2 F)(s)^2 &= G''(s)^2 \frac{(s_i s_j)^2}{|s|^4}
+ G'(s)^2 \left(\frac{\delta_{ij} |s|^2 - s_i s_j}{|s|^3}\right)^2\\
&+ 2G''(s) G'(s) \frac{s_i s_j(\delta_{ij} |s|^2 - s_j s_j)}{|s|^5}.
\end{split}
\]
By inspection,
\begin{gather*}
\sum_{i,j} \frac{(s_i s_j)^2}{|s|^4} = 1,\quad
\sum_{i,j} \left(\frac{\delta_{ij} |s|^2 - s_i s_j}{|s|^3}\right)^2
= \frac{1}{|s|^2}\\
\sum_{i,j} \frac{s_i s_j(\delta_{ij} |s|^2 - s_j s_j)}{|s|^5}
 = 0.
\end{gather*}
Therefore
\begin{gather}
\label{eq.f2}
\no{F''(s)}_{\mathcal{L}(\R^2)} ^2\leq G''(s)^2 +
\frac{G'(s)^2}{|s|^2}\leq 2C^2 |s|\sp{2(p - 2)}\\
\label{eq.f4}
|F'(s)|\leq C|s|^{p - 1}
\end{gather}
In the next proposition, the inequality
\begin{equation}
\label{eq.2}
(a + b)\sp c\leq 2\sp{c - 1}(a^c + b^c),\quad a,b\geq 0
\end{equation}
will be applied several times with different positive
values of $ c $.
\begin{proposition}
\label{prop.regular}
Let $ G $ be such that \eqref{eq.g2} is satisfied. Then $ E $ is differentiable.
If \eqref{eq.g4} holds too, then $ E $ is two-times differentiable.
\end{proposition}
\begin{proof}
(i). We will use the notation $ E $ for the non-linear part
$ \nint F(u)dx $.
We have to prove that
\[
E(u_0 + h) - E(u_0) - \nint F'(u_0)\cdot h = o(\no{h}_{H^1}).
\]
In fact, the term on the left is
\[
A := \nint \int_0\sp 1 (F'(u_0 + th) - F'(u_0))\cdot h\,dt dx.
\]
We set
\[
h^*(x) := \int_0\sp 1 |F'(u_0(x) + th(x)) - F'(u_0(x))|dt
\]
From \eqref{eq.f2} and \eqref{eq.2}
there exists $ C_1 = C_1(C,p) $ such that
\[
|F'(u_0 + th) - F'(u_0)|\leq
C_1(|u_0|\sp{p - 1} + t\sp{p - 1} |h|\sp{p - 1}).
\]
After the integration on the interval $ [0,1] $, we obtain
\[
h^* (x)\sp{p/(p - 1)}\leq C_2\left(|u_0(x)|\sp p +
|h(x)|\sp p\right).
\]
We prove that $ A = o(\no{h}) $ with a contradiction
argument. If it is false, there exists $ \var_0 > 0 $
and a sequence $ (h_n) $ converging to zero in $ H^1 $
such that
\begin{equation}
\label{eq.1}
|A|\geq\var_0 \no{h_n}.
\end{equation}
Up to extract a subsequence, we can suppose that $ |h_n| $
converges to zero pointwise and it is dominated by an $ H^1 $
function $ h_0 $. Then
\[
h^* _n (x)\sp{p/(p - 1)}\leq C_2 (|u_0 (x)|\sp p +
h_0 (x)\sp p).
\]
We have a dominated sequence in $ L^1 (\R) $ converging
pointwise a.e. to zero. Therefore,
\[
\nint |h^* _n (x)|\sp{p/(p - 1)} dx\to 0.
\]
From the H\"older inequality (with exponents $ p $ and
$ p/(p - 1) $ and functions $ h_n $ and $ h_n * $), and
\eqref{eq.SGN},
\[
|A|\leq C_2\no{h_n}_p \nint |h^* _n (x)|\sp{p/(p - 1)}\leq
C_3 \no{h_n}_{H^1} \nint |h^* _n (x)|\sp{p/(p - 1)}.
\]
Then $ |A|\leq \no{h_n} o(1) $, giving a contradiction
with \eqref{eq.1}. Therefore, $ E $ is differentiable and
\[
\langle dE(u_0),k\rangle = \nint F'(u_0 (x))\cdot k(x)dx.
\]
(ii). If \eqref{eq.g2} and \eqref{eq.g4} hold, then $ E $ is two-times
differentiable. We set
\[
B := dE(u_0 + h) - dE(u_0) - \nint F''(u_0(x)) h(x) dx
\]
and prove that $ B = o(\no{h}_{H^1}) $. In fact,
\[
B =
\nint\int_0 ^1 \big(F''(u_0 (x) + th(x)) - F''(u_0(x))\big)
h(x) dt.
\]
Let $ k $ be an arbitrary vector of $ H^1 $. Then
\[
|\langle B,k\rangle|\leq \nint |h(x)| |k(x)| h^{**} (x) dx
\]
where
\[
h^{**} (x) := \int_0 ^1 \no{F''(u_0 (x) + th(x)) - F''(u_0(x))}%
_{\mathcal{L}(\R^2)} dt.
\]
From \eqref{eq.f4} and \eqref{eq.2}, there
exists $ C_1 ' = C_1' (C,p) $ such that
\[
\no{F''(u_0 + th) - F''(u_0)}_{\mathcal{L}(\R^2)}
\leq C_1' (|u_0(x)|^{p - 2} +  t^{p - 2} |h(x)|\sp{p - 2}).
\]
After the integration on $ [0,1] $, from \eqref{eq.2} we obtain
\[
h^{**} (x)\sp{p/(p - 2)}\leq C_2 '
(|u_0(x)|^p + |h(x)|\sp p).
\]
Now, suppose that there exists $ \var_0 > 0 $ and
a dominated sequence $ (h_n) $ converging in $ H^1 $,
pointwise a.e. to zero, such that
\begin{equation}
\label{eq.3}
\no{B}_{\mathcal{L}(H^1,\R)}\geq\var_0\no{h_n}_{H^1}.
\end{equation}
By the H\"older inequality and \eqref{eq.SGN}
\[
|\langle B,k\rangle|\leq C_2'
\no{h_n}_p \no{k}_p \no{h_n ^{**}}_{p/(p - 2)}\leq
C_3' \no{h_n}\no{k}\no{h_n ^{**}}_{p/(p - 2)}
\]
for every $ k\in H^1 $. Then,
\[
\no{B}_{\mathcal{L}(H^1,\R)}\leq C_3'
\no{h_n}\no{h_n ^{**}}_{p/(p - 2)} = \no{h_n} o(1)
\]
which gives a contradiction with \eqref{eq.3}.
\end{proof}
\def\cprime{$'$} \def\cprime{$'$} \def\cprime{$'$} \def\cprime{$'$}
  \def\cprime{$'$} \def\cprime{$'$} \def\cprime{$'$} \def\cprime{$'$}
  \def\cprime{$'$} \def\polhk#1{\setbox0=\hbox{#1}{\ooalign{\hidewidth
  \lower1.5ex\hbox{`}\hidewidth\crcr\unhbox0}}} \def\cprime{$'$}
  \def\cprime{$'$} \def\cprime{$'$}
\providecommand{\bysame}{\leavevmode\hbox to3em{\hrulefill}\thinspace}
\providecommand{\MR}{\relax\ifhmode\unskip\space\fi MR }
\providecommand{\MRhref}[2]{%
  \href{http://www.ams.org/mathscinet-getitem?mr=#1}{#2}
}
\providecommand{\href}[2]{#2}

\nocite{Son10,TVZ07,BBGM07,BJM09}
\thispagestyle{empty}
\end{document}